\documentclass[reqno,10pt]{amsart}

\usepackage{amssymb,amsmath,amsthm,amsfonts}
\usepackage{mathrsfs,dsfont, comment,mathscinet}
\usepackage{enumerate,esint}
\usepackage[colorlinks=true, linkcolor=blue, citecolor=blue, urlcolor=blue]{hyperref}
\usepackage[left=2.8cm,right=2.8cm,top=2.8cm,bottom=2.8cm]{geometry}
\usepackage{mathtools}
\usepackage[pdftex]{color, graphicx}
\usepackage{bbm}
\usepackage[english]{babel}
\usepackage{comment} 
\usepackage{esint} 
\usepackage{pdfsync}
\usepackage{enumitem}
\usepackage{caption}    
\usepackage{epstopdf}
\allowdisplaybreaks

\numberwithin{equation}{section}
\theoremstyle{plain}

\newtheorem{thm}{Theorem}[section]
\newtheorem{prop}[thm]{Proposition}
\newtheorem{lemma}[thm]{Lemma}

\newtheorem{definition}[thm]{Definition}
\theoremstyle{definition}
\newtheorem{remark}[thm]{Remark}

\newcommand{\N}{\mathbb{N}}

\newcommand{\R}{\mathbb{R}}
\newcommand{\e}{\varepsilon}
\newcommand{\dw}{\mathrm{d_W}}

\newcommand{\dx}{\,\mathrm{d}x}
\newcommand{\dy}{\,\mathrm{d}y}
\newcommand{\dif}{\,\mathrm{d}}

\newcommand{\ho}{{\mathcal H}^0}

\newcommand{\restr}{%
  \,\raisebox{-.127ex}{\reflectbox{\rotatebox[origin=br]{-90}{$\lnot$}}}\,%
}

\definecolor{pinegreen}{rgb}{0.0, 0.47, 0.44} 

\title{Quasi-crystalline phase separation in one dimension}
\author{Riccardo Cristoferi}
\address{Department of Mathematics, IMAPP, Radboud University, Heyendaalseweg, 6525 AJ 
Nijmegen, The Netherlands}
\email{cristoferi@science.ru.nl}

\author{Lorenza D'Elia}
\address{Institute of Analysis and Scientific Computing, TU Wien, Wiedner Hauptstraße 8-10, 1040
Vienna, Austria}
\email{lorenza.delia@tuwien.ac.at}
\date{}

\begin{document}

\begin{abstract}
We investigate a Modica–Mortola energy in heterogeneous materials with a quasi-crystalline microstructure, described by a suitable notion of almost periodicity. In the regime where phase separation occurs at a finer scale than homogenization, we establish compactness and $\Gamma$-convergence results in one dimension for vector-valued phases. The main novelty lies in the explicit characterization of the effective surface energy, which reveals the interplay between phase separation and homogenization.
In particular, we prove that heterogeneities can promote the splitting of a phase separation into several energetically favorable microscopic transitions.
\end{abstract}

\maketitle


\section{Introduction}

The aim of this paper is to initiate the study of variational models arising in the gradient theory of liquid–liquid phase transitions in composite materials with small-scale quasi-crystalline microstructure.
The main goal is to rigorously understand how a quasi-crystalline material can influence both the location and the energetic pathway of a phase transition.

Quasi-crystals are ordered structures showing a forbidden symmetry for periodic structures, such as $5$-, $7$-, $8$- and $10$-fold symmetries. Although the  discovery of quasi-crystals \cite{LS84,SBDG}  has been welcomed with scepticism by the scientific community, quasi-crystalline materials  have drawn wide attention in recent years thanks to their unique  proprieties: 
low electrical and thermal conductivity, low surface energy, low coefficient of friction, high hardness, as well as high brittleness. All these proprieties make quasi-crystals particularly attractive in many active fields of engineering as  photonic quasi-crystals \cite{VNA13}, catalysis \cite{TKT06} and optics \cite{AATAKE22}.
 The research of production methods of materials forming a  quasi-crystalline phase has attracted extensive attention in many field of engineering. Among such materials, aluminium-copper-iron alloys (see \cite{HS04} for an overview) arouse a great interest  due to the lack of toxicity and the easy availability of the alloying elements and they find  application in many areas such as coating materials \cite{FGSLB18} and catalysis \cite{TKT06}.  
Although many quasi-crystals, like aluminium alloys,  are synthetically obtained,  scientists have recently also found them available in nature \cite{BYal15}. 

An analytical approach to the study of quasi-crystalline material relies on the notion of \emph{quasi-periodic functions}, representing the very first generalization of the notion of periodicity. Such a-periodic functions can be defined as a section of a periodic map of more variables and are also characterized by Fourier series with finitely many frequencies. The closure of such a class with respect to uniform norm gives rise to a broader one, the so-called \emph{almost periodic functions}. The theory of this class has been developed by Bohr \cite{Bohr25} at the beginning of the last century. Starting from the seminal paper \cite{Bohr25}, several generalizations have been proposed,  each of which can be obtained as  closure of the set of trigonometric polynomials with respect to some suitable metric (see, e.g.,  \cite{B32, St25}).
Beyond the harmonic analysis, representing the very first field to have been investigated, almost periodic functions play a crucial role in the theory of partial differential equations \cite{LZ98}, in celestial mechanics \cite{F74}, as well as in statistical mechanics \cite{CBL98}.

Now, we give a brief overview of the mathematical analysis of the gradient theory of phase separation in the general framework.
Let $\Omega\subset\R^N$ be an open bounded set with Lipschitz boundary, and let $M\geq 1$.
We consider the model by Van der Walls \cite{VdW}, Cahn and Hillard  \cite{CH58} for homogeneous materials
\[
G_\varepsilon(u) \coloneqq \int_\Omega \left[ \frac{1}{\varepsilon}W(u(x)) + \varepsilon|\nabla u(x)|^2 \right] \mathrm{d} x.
\]
Here, $u\in H^1(\Omega;\R^M)$ is the phase variable, $W: \R^M\to[0,\infty)$ is the potential vanishing at the wells $\{z_1, \dots, z_k\}$, which represents the stable phases.
The parameter $\varepsilon>0$ is related to the scale of the transition region between bulk regions where the phase variable is close to the stable phases $\{z_1, \dots, z_k\}$, and it is usually very small.
The main focus of the analysis is in rigorously understanding of the behavior of the functional $F_\varepsilon$ when $\varepsilon$ is very small by using the language of $\Gamma$-convergence (for an introduction on this topic see \cite{Dalmasobook}).
It has been proved by several authors (see \cite{baldo}, \cite{FonsecaTartar}, \cite{Gurtin}, \cite{Stern}) that the limiting functional is of the form
\begin{equation}\label{eq:F_00_homog}
G_0(u)\coloneqq \sum_{i\neq j=1}^k \sigma_{ij} \,\mathcal{H}^{N-1}(J^{ij}_u\cap \Omega),
\end{equation}
for functions $u\in BV(\Omega;\{z_1,\dots,z_k\})$, namely functions of bounded variations taking values only in $\{z_1,\dots,z_k\}$, and is equal to $+\infty$ else in $L^1(\Omega;\R^M)$.
Here, $\mathcal{H}^{N-1}$ denotes the $(N-1)$-dimensional Hausdorff measure, $J^{ij}_u$ the set of points where $u$ jumps from $z_i$ to $z_j$, and the limiting surface energy density $\sigma_{ij}>0$ is given by
\begin{equation}\label{eq:sigma_intro}
\sigma_{ij}\coloneqq \inf\left\{\int_{-1}^1 2\sqrt{W(\gamma(t))}|\gamma'(t)|\mathrm{d} t \,:\, \gamma\in W^{1,1}([-1,1];\R^M), \gamma(-1)=z_i, \gamma(1)=z_j \right\}.
\end{equation}
This latter represents the optimal way to transition from the phase $z_i$ and the phase $z_j$, and is given by a degenerate geodesic problem.

In the case the material has a periodic microstructure, the functional $G_\varepsilon$ can be modified as follows
\begin{equation}\label{eq:F_periodic}
H_{\e,\delta}(u)\coloneqq \int_\Omega \left[ \frac{1}{\e} W\left(\frac{x}{\delta}, u(x) \right) + \e|\nabla u(x)|^2 \right] \dx,
\end{equation}
where now $W:\R^N\times\R^M\to[0,\infty)$ is $[0,1)^N$-periodic in the first variable.
The interaction between the periodic homogenization and the phase separation leads to considering the three regimes 
\[
\varepsilon\ll\delta, \;\;\;\; \e\sim\delta, \;\;\;\; \delta\ll \varepsilon.
\]
In all three cases, the limiting energy density is of the form (see \cite{CFG_moving_sub}, \cite{CFG_fixed_sub}, \cite{CFG_fixed_sup},  \cite{CriFonHagPop})
\[
H_{0,0}(u)\coloneqq \sum_{i\neq j=1}^k\int_{J^{ij}_u\cap \Omega} \sigma_{ij}(\nu(x)) \mathrm{d} \mathcal{H}^{N-1}(x),
\]
where $\nu(x)\in\mathbb{S}^{N-1}$ is the measure theoretic outer normal at the point $x$, and the function $\nu\mapsto \sigma_{ij}(\nu)$ depends on the regime under consideration, and generalises the surface energy density in \eqref{eq:sigma_intro}.\\

\par  In this manuscript, we adopt the model \eqref{eq:F_periodic} in the one-dimensional case, namely when $\Omega$ is an open bounded interval of the form $(a,b)$, with $a<b$. 
For $u\in H^1(\Omega;\R^M)$ we define
\begin{equation}
         \notag
         F_{\e,\delta}(u)\coloneqq \int_\Omega \left[ \frac{1}{\e} W\left(\frac{x}{\delta}, u(x) \right) + \e|u'(x)|^2 \right] \dx,
     \end{equation}
where now the potential $W:\R\times\R^M\to[0,\infty)$ describes a quasi-crystalline material by using the framework of a generalization of Bohr almost periodic functions (for the precise assumptions, see Section \ref{sec:mainresult}).
A possible physical motivation for the energies considered in this paper is the phase separation of multi-component liquid mixtures confined in heterogeneous porous materials. The influence of pore geometry and topology on the phase separation of binary mixtures has been investigated in \cite{ShimizuTanaka2017}, while recent experimental developments have demonstrated the possibility of constructing porous materials with prescribed quasi-crystalline microstructures \cite{ChenAroraXiangEtAl2026}.
These observations motivate the investigation of phase-transition models in materials with quasi-crystalline heterogeneities. In our framework, the underlying microstructure is encoded through the spatial dependence of an almost-periodic potential. The assumption of spatially independent wells corresponds to a situation in which the preferred equilibrium phases remain unchanged throughout the material, while the energetic cost of transitions between them depends on the local microstructure. For multi-component mixtures, such heterogeneities may affect different transition pathways in distinct ways, naturally motivating the study of vector-valued order parameters and general, non multiplicative multi-well potentials.

As for the case of periodic materials, we expect the asymptotic behavior of the functional $F_{\e,\delta}$ to depend on the regime
\[
\varepsilon\ll\delta, \;\;\;\; \e\sim\delta, \;\;\;\; \delta\ll \varepsilon.
\]
Here, we focus on the first regime, namely when the scale of the heterogeneities is larger than that at which phase separation occurs. The other regimes are currently under investigation.
Restricting our analysis to the one-dimensional setting allows us to treat vector-valued phases with a general potential $W$. Indeed, in higher dimensions, even in the periodic setting, existing results require the potential $W$ to have a multiplicative structure and are restricted to the scalar case $M=1$ (see \cite{CFG_fixed_sub}).

The main results of the present paper are the compactness (see Theorem \ref{thm:compactness}) and the characterization in the language of $\Gamma$-convergence of the effective surface energy $F_{0,0}$ in the regime $\varepsilon\ll\delta$ (see Theorem \ref{thm:main}).
The latter represents the first novelty of this contribution. We show that (compare with \eqref{eq:F_00_homog})
	\begin{equation}
		\notag
		F_{0,0}(u)\coloneqq \sum_{i\neq j =1}^k \sigma_{ij} \ho(J^{ij}_u \cap \Omega),
	\end{equation}
	for functions $u\in BV(\Omega;\{z_1,\dots, z_k\})$, namely piecewise constant with taking values in $\{z_1,\dots, z_k\}$.
    Here, $J^{ij}_u$ represents the points where $u$ jumps from $z_i$ to $z_j$, and
    \[
    \sigma_{ij}\coloneqq \inf\left\{ \int_{J_w} \dw(x; w^{-}(x), w^{+}(x)) \dif{\mathcal H}^0(x) 
        \,:\, w\in \mathcal{A}_{i,j}  \right\},
    \]
    where $\dw(x; w^{-}(x), w^{+}(x))$ is the geodesic distance of the potential $W(x,\cdot)$ (see \eqref{eq:sigma_intro}), and $\mathcal{A}_{i,j}$ is the admissible class of competitors taking values in $\{z_1,\dots, z_k\}$ and equal to the jump function
    \[
    \widetilde{w} (x) \coloneqq
    \begin{cases}
        z_i & x\leq0,\\
        z_j & x> 0,
    \end{cases}
    \]
    outside of a bounded interval (see Definition \ref{def:limiting_density} for more details).
    While explicit characterizations of the effective surface energy have previously been obtained in the periodic setting for scalar-valued phases and potentials with a multiplicative structure (see \cite{CFG_fixed_sub}), our analysis allows for general multiwell potentials and vector-valued phases. Moreover, we establish these results in the broader framework of Bohr almost-periodic media, where the lack of a periodic unfolding operator requires a substantially different approach to the liminf inequality.
    In particular, the explicit characterization of the limiting surface energy density $\sigma_{ij}$ provides a clear interpretation of the interplay between the geometry of the underlying material and the phase separation phenomenon.
First, we observe a \emph{separation of scales}: the limiting functional $F_{0,0}$ coincides with the one obtained by first letting $\varepsilon\to 0$, corresponding to phase separation, and subsequently letting $\delta\to 0$, corresponding to homogenization.
In the latter process, at a jump point $x_0\in J^{ij}_u$, rather than considering a single optimal profile associated with the geodesic distance $\dw(x_0; z_i, z_j)$, as in \eqref{eq:sigma_intro}, it may be energetically favorable to consider \emph{multiple optimal profiles at different points} $x\in\R$, each corresponding to the geodesic distance $\dw(x; z_i, z_j)$.

The main technical challenge of the paper lies in establishing the liminf inequality, which involves first passing to the limit as $\varepsilon\to 0$ and subsequently as $\delta\to 0$.
Indeed, previous approaches in the periodic setting rely crucially on the unfolding operator, which is not available in the present framework. This calls for a tailored strategy to establish the liminf inequality. To this end, we introduce an auxiliary potential depending on a small parameter $\eta>0$, which enables us to prove that the limiting surface energy density is independent of the spatial variable. The construction of this auxiliary potential is precisely where the assumption of Bohr almost periodicity plays a fundamental role, as it provides the \emph{pointwise} control needed to describe the behavior of the potential over increasingly large intervals.
Roughly speaking, this allows us to reduce the problem to the study of a Modica–Mortola functional on the whole real line $\R$, without rapid spatial oscillations. We can then invoke recent results on such heterogeneous functionals to derive an effective formula for the auxiliary potential. Finally, letting $\eta\to 0$, we recover the desired limiting surface energy density.

The construction of the recovery sequence requires an adaptation of the arguments of \cite{CriGra} for space-dependent potentials.
In particular, the one dimensionality of the domain simplifies greatly the argument, since we do not need to deal with general sets of finite perimeter.\\

This paper is organized as follows: in Section \ref{sect:setup} we introduce the set-up of the problem and the statements of the main results. In Section \ref{sec:Technical Results} we collect all the techinical results which we will use to prove the main results. In Section \ref{sec:Compactness} we prove the compactness, while the proof of the liminf equality is given in Section \ref{sect:liminf}. The recovery sequence is constructed in Section \ref{sect:limsup}.
Finally, our arguments are stable enough to be applied also to the case where a mass-constraint is imposed and which is treated in Section \ref{sect:mass-constrained}.


\subsection{Literature review}

    The literature on the gradient theory of phase separation is very large.
    Here, we provide a selection of contributions that are closed to the investigation of the present manuscript.
    For a more extensive overview, see the introduction of \cite{CriGra}.
    The case of periodic materials where the oscillation is at the level of the potential has been considered by the first author together with Fonseca, Hagerty, and Popovici in \cite{CriFonHagPop} (see also \cite{Hagerty}) and by the first author together with Fonseca and Ganedi \cite{CFG_fixed_sup} for the case of fixed wells in the regime $\varepsilon\sim\delta$ (see also \cite{Don}), and $\delta\ll\varepsilon$, respectively.
    Moreover, in \cite{CFG_moving_sub} there is a study the regime $\varepsilon\ll\delta$ where the wells of the potential are space dependent.
    
    The case of periodic potentials with multiple scales is considered in the recent contribution \cite{CriPig} by the first author and Pignatelli, and in \cite{Pig} by the latter alone.

    The case where the high oscillations are in both the potential and the singular perturbation has been considered by Morfe and Wagner in \cite{MorWag}, in the regime $\delta\ll\varepsilon$ in the case the potential depends in a multiplicative way from the space variable. In this work, the authors are able to consider the case of quasi-periodic functions.

    In \cite{AnsBraChi1} and \cite{AnsBraChi2}, Ansini, Braides, and Chiad\`{o}-Piat investigated the asymptotic of the energy in the case the oscillations are the level of the singular perturbation in all the three regimes, under some special assumptions.

    Finally, the gradient flow of the energy has been recently considered by the first author together with Ganedi and Pignatelli in \cite{CriGanPig} in the regime $\delta\ll\varepsilon^{3/2}$.
    Moreover, Morfe in \cite{Mor} considered the case where the oscillations are at the level of the singular perturbation in the regime $\varepsilon\ll\delta$.


\section{Set-up of the problem and main result}
\label{sect:setup}

\subsection{Main results}\label{sec:mainresult}
For $a<b$, let $\Omega\coloneqq (a,b)\subset\R$ be a bounded open interval.
We first introduce the terminology for the long-range regularity of the potential.
This is an adaptation of the usual notion of Bohr almost periodicity, where we also impose a uniformity condition on the relatively dense set of almost periods with respect to the second variable. For more details on almost periodic functions see \cite{ABG06, LZ82} and references therein.

\begin{definition}
    A set $X\subset \R$ is said to be \emph{relatively dense} if there exists a number $\ell>0$ such that every interval $[a, a+\ell]$ contains at least one point of $X.$ The number $\ell$ is called the \emph{inclusion interval}.
\end{definition}

\begin{definition} \label{def:Bohralmper}
   We say that the function $f\in C^0(\R\times\R^M; \R)$ is \emph{uniformly almost-periodic} if for any compact set $K\subset\R^M$ and for every $\lambda>0$, there exists a relatively dense set $\{\tau_i\}_{i\in I_\lambda}\subset\R$, where $I_\lambda$ is a set of indexes, such that 
         \begin{equation}
             \notag
             \sup_{x\in\R} |f(x+\tau_i,p)-f(x,p)|<\lambda,
         \end{equation}
    for any $i\in I_\lambda$ and for any $p\in K$. Each $\tau_i$, for $i\in I_\lambda$, is called an \emph{$\lambda$-uniformly almost period} of $f$.
\end{definition}

\begin{remark}
The class of functions introduced in Definition \ref{def:Bohralmper} includes both periodic and quasiperiodic functions, the latter being functions that can be represented as periodic functions in a higher-dimensional space (see, e.g.,  \cite{ABG06}).
\end{remark}

We are now in position to state all of the assumptions that we will require in this manuscript.
Let $W:\R\times\R^M\to[0,\infty)$, with $M\geq 1$ satisfying the following:
    \begin{enumerate}[label=(H\arabic*), ref=(H\arabic*)]
	\item \label{H1} $W$ is uniformly almost-periodic;
	\item \label{H2} There exist $z_1, \dots, z_k\in \R^M$ such that
	\[
	W(x,p)=0 \quad \text{if and only if }\quad p\in\{z_1, \dots z_k\},
	\]
	for a.e. $x\in\R$;
	\item \label{H3} There exists $R>0$ such that
	\[
	W(x,z)\geq \frac{|z|}{R},
	\]
	for any $x\in\R$, and for all $z\in\R^M$ with $|z|\geq R$;
    \item \label{H4} There exists $V: \R^M\to [0, +\infty)$ with $V^{-1}(0)=\{z_1,\dots,z_k\}$ and with at least linear growth at infinity such that
    \[
    W(x, p) \geq V(p)
    \]
    for any $x\in \R$ and for any $p\in\R^M$;
    \item \label{H5} If $M>1$, then there exist a constant $C_1>0$ and a continuous non-decreasing function $f\colon[0,\infty)\to[0,\infty)$ with $f^{-1}(\{0\})=0$ and $f(2t)\leq C_1 f(t)$ for all $t\geq 0$ such that
            \[
            \frac{1}{C_1} f\left(\min\{ |p-z_i| : i=1\dots,k \}\right)
            \leq W(x,p)
                \leq C_1 f\left(\min\{ |p-z_i| : i=1\dots,k \}\right),
            \]
            for a.e. $x\in\Omega$, and all $p\in \R^M$;
	\end{enumerate}

\begin{remark}
Assumptions \ref{H3} and \ref{H4} are standard in order to ensure compactness. The former is to make sure that the potential is uniformly with at most linear growth at infinity.
The latter asks for a non-degeneracy of the potential in the domain $\Omega$.
\end{remark}

\begin{remark}
    Assumption \ref{H5} was considered in \cite{CriDeuPig} as a low regularity assumption on the behavior of the potential close to the wells.
    It is needed in order to ensure that geodesics of the degenerate metric \eqref{eq:distance_dw} have finite Euclidean length. 
    Note that, in the case $M=1$, then the length of the geodesics between $z_i$ and $z_j$ is $|z_i-z_j|$.
\end{remark}

We now introduce the family of functionals that we will study.

\begin{definition}
For $\e,\delta>0$, let $F_{\e,\delta}: L^1(\Omega;\R^M)\to \R\cup \{+\infty\}$ be the functional defined as
     \begin{equation}
         \notag
         F_{\e,\delta}(u)\coloneqq \int_\Omega \left[ \frac{1}{\e} W\left(\frac{x}{\delta}, u(x) \right) + \e|u'(x)|^2 \right] \dx,
     \end{equation}
if $u\in H^1(\Omega;\R^M)$, and as $F_{\e,\delta}(u)\coloneqq+\infty$ otherwise in $L^1(\Omega;\R^M)$.
\end{definition}

We now introduce the limiting functional. 
We first introduce its domain $BV(\Omega;\{z_1,\dots, z_k\})$. For further details on BV spaces see \cite{AFP, L17}.

\begin{definition}\label{def:limiting_density}
Let $\Omega\subset\R$ be a bounded open set.
Let $z_1,\dots, z_k\in\R^M$. We define the space $BV(\Omega;\{z_1,\dots, z_k\})$ as the space of functions $u\in L^1(\Omega;\R^M)$ with $u(x)\in\{z_1,\dots, z_k\}$ for a.e. $x\in\Omega$, and such that, for each $i\in\{1,\dots, k\}$, the set $\{x\in\Omega : u(x)=z_i\}$ is a finite union of open intervals.
\end{definition}

In the following, we will denote by $\mathcal{H}^0$ the zero-dimensional Hausdorff measure. Next, we define the limiting energy density.

\begin{definition}
For $x\in\R$, and $p,q\in\R^M$, we define
\begin{equation}\label{eq:distance_dw}
\dw(x;p,q) \coloneqq \inf\left\{ 2\int_{-1}^1 \sqrt{W(x,\gamma(t))}|\gamma'(t)|\, \dif t \,:\, \gamma\in W^{1,1}([-1,1];\R^M),
    \gamma(-1)=p,\, \gamma(1)=q \right\}.
\end{equation}
For $i\neq j\in\{1,\dots,k\}$, we define
    \[
    \sigma_{ij}\coloneqq \inf\left\{ \int_{J_w} \dw(x; w^{-}(x), w^{+}(x)) \dif{\mathcal H}^0(x) 
        \,:\, w\in \mathcal{A}_{i,j}  \right\},
    \]
where the admissible class $\mathcal{A}_{i,j}$ is defined as 
\begin{equation}
    \notag
    {\mathcal A}_{i,j}\coloneqq \{w\in BV(\R; \{z_1, \dots, z_k\}) : \exists\, \overline{t}, \overline{s}\in\R \;\; \hbox{s.t. } w(z)= z_i\; \hbox{for } z\leq \overline{t}\;\; \hbox{and}\;\; w(z)= z_j\; \hbox{for } z\leq \overline{s}\},
\end{equation}
for all $i\neq j\in\{1,\dots,k\}$.
\end{definition}

\begin{remark}
Note that in the case where the potential $W$ is $[0,1)^N$-periodic in the first variable, we can restrict the class of admissible competitors for $\sigma_{ij}$ to ${\mathcal A}_{i,j}\cap BV(Q; \{z_1, \dots, z_k\})$.
\end{remark}

We are now in position to define the object of interest.

\begin{definition}
Let $F_{0,0}: L^1(\Omega;\R^M)\to \R\cup \{+\infty\}$ be defined as
	\begin{equation}
		\notag
		F_{0,0}(u)\coloneqq \sum_{i\neq j =1}^k \sigma_{ij} \ho(J^{ij}_u \cap \Omega),
	\end{equation}
	if $u\in BV(\Omega;\{z_1,\dots, z_k\})$, and as $F_{0,0}(u)\coloneqq+\infty$ otherwise in $L^1(\Omega;\R^M)$.
    Here, we use the notation
    \[
    J^{ij}_u \coloneqq \{ x\in J_u \,:\, u^-(x)= z_i,\, u^+(x)=z_j \}
    \]
    to denote the points of the jump set of $u$ having left trace $z_i$ and right trace $z_j$.
\end{definition}

\begin{remark}
Note that we do not need a factor $1/2$ in front of the sum in the definition of the functional $F_{0,0}$, since $J^{ij}_u\cap J^{ji}_u=\emptyset$ for $i\neq j$.
\end{remark}

First of all, we state the pre-compactness of sequences with uniformly bounded energy.

\begin{thm}\label{thm:compactness}
Let $(\varepsilon_n)_n, (\delta_n)_n\subset(0,\infty)$ be two infinitesimal sequences such that
\[
\lim_{n\to\infty} \frac{\varepsilon_n}{\delta_n}=0.
\]
Assume that {\rm \ref{H1}}-{\rm \ref{H4}} are in force.
Let $(u_n)_n\subset H^1(\Omega;\R^M)$ be such that
\[
\sup_{n\in\N} F_{\varepsilon_n,\delta_n}(u_n) < +\infty.
\]
Then, up to a subsequence, it holds that $u_n\to u$ strongly in $L^1(\Omega;\R^M)$, for some $u\in BV(\Omega;\{z_1,\dots,z_k\})$.
\end{thm}

The main result of this paper is the characterization of the asymptotic limit of the family $F_{\e,\delta}$ as $\e,\delta\to0$ in the regime $\e\ll\delta$.

\begin{thm}\label{thm:main}
Let $(\varepsilon_n)_n, (\delta_n)_n\subset(0,\infty)$ be two infinitesimal sequences such that
\[
\lim_{n\to\infty} \frac{\varepsilon_n}{\delta_n}=0.
\]
Assume that {\rm \ref{H1}}-{\rm \ref{H5}} are in force.
Then, $(F_{\e,\delta})_{\e,\delta}$ $\Gamma$-converges to $F_{0,0}$ with respect to the strong $L^1(\Omega;\R^M)$ topology.
\end{thm}

The proofs we will develop can be easily adapted in order to also cover the mass constrained case.
Nevertheless, as it is customary in this case, extra requirements on the behavior of the potential close to the wells are needed (see \cite{FonPopo}, \cite{CriDeuPig}).

\begin{enumerate}[label=(H\arabic*), ref=(H\arabic*), resume]
\item \label{H7} There exist $r_0 > 0$ and a non-decreasing function $\theta: [0,r_0] \to [0,+\infty)$ with $\lim_{r \to 0^+} \theta (r) = 0$ such that, for all $x\in\R$, it holds
    \[
    W(x,p) \leq \theta(|p-z_i|) |p-z_i|,
    \]
    for all $p \in \R^M$ with $|p-z_i|\leq r_0$.
\end{enumerate}

We now introduce the mass-constrained functionals.

\begin{definition}
Fix $m\in\R^M$. For $\e,\delta>0$, let $F^m_{\e,\delta}: L^1(\Omega;\R^M)\to \R\cup \{+\infty\}$ as
\[
F^m_{\e,\delta}(u) \coloneqq 
\left\{
\begin{array}{ll}
\displaystyle F_{\e,\delta}(u) & \text{ if } u\in H^1(\Omega;\R^M), \quad \int_\Omega u(x) \dx = m,\\
& \\
+\infty & \text{ else}.
\end{array}
\right.
\]
Moreover, we define $F^m_{0,0}: L^1(\Omega;\R^M)\to \R\cup \{+\infty\}$ as
\[
F^m_{0,0}(u) \coloneqq 
\left\{
\begin{array}{ll}
\displaystyle F_{0,0}(u) & \text{ if } u\in BV(\Omega;\{z_1,\dots,z_k\}), \quad \int_\Omega u(x) \dx = m,\\
& \\
+\infty & \text{ else}.
\end{array}
\right.
\]
\end{definition}

\begin{thm}\label{thm:mass_constrained}
Fix $m\in\R^M$. Let $(\varepsilon_n)_n, (\delta_n)_n\subset(0,\infty)$ be two infinitesimal sequences such that
\[
\lim_{n\to\infty} \frac{\varepsilon_n}{\delta_n}=0.
\]
Assume that {\rm \ref{H1}}-{\rm \ref{H4}} are in force.
Let $(u_n)_n\subset H^1(\Omega;\R^M)$ be such that
\[
\sup_{n\in\N} F^m_{\varepsilon_n,\delta_n}(u_n) < +\infty.
\]
Then, up to a subsequence, it holds that $u_n\to u$ strongly in $L^1(\Omega;\R^M)$, for some $u\in BV(\Omega;\{z_1,\dots,z_k\})$ with
\[
\int_\Omega u(x) \dx = m.
\]
Moreover, if {\rm \ref{H1}}-{\rm \ref{H7}} hold, then $(F^m_{\e,\delta})_{\e,\delta}$ $\Gamma$-converges to $F^m_{0,0}$ with respect to the strong $L^1(\Omega;\R^M)$ topology.
\end{thm}


\section{Technical Results}
\label{sec:Technical Results}

In this section, we collect the several technical results that will be needed in the proof of the main result (Theorem \ref{thm:main}).
For the reader's convenience, we divide them in separare sections.

\subsection{Previous results on the Modica-Mortola functional}

First, we state a classical reparametrization lemma due to Modica (see \cite[roposition 2]{modica87}).
For a proof, see \cite[Lemma 4.5]{CriGra}.

\begin{lemma}\label{lemma:reparameterization}
    Fix $\lambda > 0$, $\varepsilon > 0$, $x \in \Omega$, and $p,q \in \R^M$. Let $\gamma \in C^1([-1,1]; \R^M)$ with $\gamma'(s) \neq 0$ for all $s \in (0,1)$, $\gamma(-1) = p$, $\gamma(1) = q$. Then, there exist $\nu > 0$ and $C > 0$ with
    \begin{equation}\label{eq:nu}
    C \varepsilon \leq \nu \leq \frac{\varepsilon}{\sqrt{\lambda}} \int_{-1}^1 | \gamma'(s)| \dif s,
    \end{equation}
    and $g \in C^1([0, \nu]; [-1,1])$ such that
    \begin{equation}\label{eq:g}
        (g'(t))^2 = \frac{\lambda + W(x,\gamma(g(t)))}{\varepsilon^2 | \gamma'(g(t))|^2},
    \end{equation}
    for all $t \in (0,\tau)$, $g(0) = 0$, $g(\tau) = 1$, and
    \begin{align*}
        \int_{-1}^\nu \left[ \frac{1}{\varepsilon} W(x, \gamma(g(t))) + \varepsilon |\gamma'(g(t))|^2 | g'(t)|^2 \right] \dif t
        \leq \int_{-1}^1 2 \sqrt{W(x, \gamma(s))} | \gamma'(s) | \dif s
            + 2 \sqrt{\lambda} \int_{-1}^1 | \gamma'(s) |\dif s.
    \end{align*}
\end{lemma}

\begin{remark}
The requirement $\gamma'(s) \neq 0$ for all $s \in (0,1)$ can be easily obtained by using a reparametrization of the curve (see \cite[Remark 6.4]{CriPig}).
\end{remark}

In the liminf inequality (see Proposition \ref{prop:liminf}), we need the following result which is essentially contained in \cite{CriGra}.
Since the assumptions we use in here are slightly different, we provide a sketch of the proof, for the reader's convenience.

\begin{thm}\label{thm:mm_x}
Let $(T_k)_k\subset(0,1)$ be an infinitesimal sequence, and let $L>0$.
Let $U:(-L,L)\times \R^M\to [0,\infty)$ be a continuous potential vanishing at $\{z_1,\dots,z_k\}$, with at least linear growth at infinity and satisfying Assumption {\rm \ref{H4}}.
Let $(u_k)_k \subset H^1((-L,L);\R^M)$ be such that
\[
\sup_{k\in\N}\int_{(-L,L)} \left[ \frac{1}{T_k} U(x,u_k(x)) + T_k |u'_k(x)|^2  \right] \dx < \infty.
\]
Then, up to a subsequence (not relabeled) $u_k\to u$ for some $u\in BV((-L,L);\{z_1,\dots,z_k\})$.

Moreover, assume that there exists $C>0$ such that the following holds: for all $z_i\neq z_j$, if $\widetilde{\gamma}\in W^{1,1}([-1,1];\R^M)$ is such that $\widetilde{\gamma}(-1)=z_i$, $\widetilde{\gamma}(1)=z_j$ and
\begin{align}
\notag
\int_{-1}^1 2\sqrt{U(x, \widetilde{\gamma}(t))}|\widetilde{\gamma}'(t)|\mathrm{d} t
    &= \lambda_{ij}(x)\\
    &\coloneqq \min\biggl\{ \int_{-1}^1 2\sqrt{U(x, \gamma(t))}|\gamma'(t)|\mathrm{d} t \,:
    \, \gamma\in W^{1,1}([-1,1];\R^M),
        \gamma(-1)=z_i, \gamma(1)=z_j \biggr\},\notag
\end{align}
then
\[
\int_{-1}^1 |\gamma'(t)|\mathrm{d} t \leq C.
\]
Let $u\in BV((-L,L);\{z_1,\dots,z_k\})$. Then,
\[
\liminf_{k\to\infty} \int_{(-L,L)} \left[ \frac{1}{T_k} U(x,u_k(x)) + T_k |u'_k(x)|^2  \right] \dx \geq \int_{J_u\cap(-L,L)} \lambda_{ij}(x)\mathrm{d} \mathcal{H}^0(x),
\]
for all $(u_k)_k\subset H^1((-L,L);\R^M)$ with $u_k\to u$ strongly in $L^1((-L,L);\R^M)$.
\end{thm}

\begin{proof}
We sketch the proof of the lower bound.
We use blow-up at a jump point $x_0\in(-L,L)$ of $u$ between $z_i$ and $z_j$.
Let $(r_n)_n\subset(0,1)$ be an infinitesimal sequence.
Then, we have to bound from below the quantity
\[
\lim_{n\to\infty} \lim_{k\to\infty} \int_{(x_0-r_n, x_0+r_n)} \left[ \frac{1}{T_k} U(x,u_k(x)) + T_k |u'_k(x)|^2  \right] \dx.
\]
For $n\in\N$, define $F_n:\R^M\to[0,\infty)$ as
\[
F_n(p)\coloneqq \inf\{ U(x,p) : x\in (x_0-r_n, x_0+r_n) \}.
\]
Thus, we have that
\begin{align*}
\lim_{n\to\infty} \lim_{k\to\infty}& \int_{(x_0-r_n, x_0+r_n)} \left[ \frac{1}{T_k} U(x,u_k(x)) + T_k |u'_k(x)|^2  \right] \dx \\
    &\geq \lim_{n\to\infty} \lim_{k\to\infty} \int_{(x_0-r_n, x_0+r_n)} \left[ \frac{1}{T_k} F_n(u_k(x)) + T_k |u'_k(x)|^2  \right] \dx \\
    &\geq \lim_{n\to\infty} \lambda^n_{ij},
\end{align*}
where in the last step we used the classical Modica-Mortola argument for $F_n$ (see \cite{modica87}), and we set
\[
\lambda^n_{ij}\coloneqq \min\left\{ \int_{-1}^1 2\sqrt{F_n(\gamma(t))}|\gamma'(t)|\mathrm{d} t \,:\, \gamma\in W^{1,1}([-1,1];\R^M),
        \gamma(-1)=z_i, \gamma(1)=z_j \right\}.
\]
To conclude, note that $F_n$ in non-decreasing and that, by continuity of $U$ in the first variable, we have $F_n(p)\to U(x_0,p)$ pointwise as $n\to\infty$.
Thus, by Dini's Theorem, we get that $(F_n)_n$ converges to $U(x_0,\cdot)$ uniformly on every compact set of $\R^M$.
Using the growth of $U$ at infinity, we know that there exists $S>0$ such that $|\gamma(t)|\leq S$ for all geodesics for $\lambda_{ij}(x_0)$.
Therefore, we have that
\[
\lim_{n\to\infty}\lambda^n_{ij} = \lambda_{ij}(x_0),
\]
which gives the desired result.
\end{proof}


\subsection{Bound on the Euclidean length of geodesics}

Here, we study the behavior of solutions to the minimization problem defining $\dw(y;p,q)$.
In particular, we show that it admits a solution, and that each solution has Euclidean length bounded by a constant depending only on $\max\{|p|, |q|\}$.
In particular, we will use this result for the special case $p,q\in\{z_1,\dots,z_k\}$.

\begin{lemma}\label{lem:geodesics}
Assume {\rm \ref{H1}}-{\rm \ref{H5}} are in force.
For all $y\in\R$, and all $i, j\in \{1,\dots k\}$, the minimization problem defining $\dw(y;z_i,z_j)$ admits a solution.
Moreover, there exists $L>0$ such that for all $y\in\R$, it holds that
\[
\int_{-1}^1 |\gamma'(t)|\, \dif t \leq L,\quad\quad\quad
\| \gamma \|_{L^\infty([-1,1];\R^M)}\leq L,
\]
for all curves $\gamma\in W^{1,1}([-1,1];\R^M)$ such that
\[
\dw(y;p,q) = 2\int_{-1}^1 \sqrt{W(y,\gamma(t))}|\gamma'(t)|\, \dif t,
\]
and with $\gamma(-1)=z_i, \gamma(1)=z_j$.
\end{lemma}

The proof of such a result is given in \cite[Theorem 4.8]{CriDeuPig}.


\subsection{Properties of the potential}

We first state a well-known property that uniformly almost-periodic functions are uniformly continuous when restricted to a compact set in the second variable. We omit the proof since it is an easy adaptation of \cite[Property 2, Ch1]{LZ82}.

\begin{lemma}\label{lem:unif_cont}
Let $f\in C^0(\R \times \R^M ; \R)$ be uniformly almost-periodic.
Let $K\subset \R^M$ be a compact set.
Then, $f$ restricted to $\R\times K$ is uniformly continuous.
\end{lemma}

Next, we prove that the Bohr almost periodicity of $W$ is inherited by $\sqrt{W}$ and by $\dw$.
We start with the former.

\begin{lemma}\label{lem:sqrtW_Sap}
Assume {\rm \ref{H1}} holds. Then, the function $\sqrt{W}$ is Bohr almost periodic.
\end{lemma}

\begin{proof}
Fix $\lambda>0$ and a compact set $K\subset\R^M$.
Using \ref{H1}, there exists a relatively dense set $\{\tau\}_\lambda$ such that
\[
\sup_{y\in\R} |W(y+\tau,p) - W(y,p)|<\lambda^2,
\]
for any $\tau\in\{\tau\}_\lambda$ and for any $p\in K$.
Fix $y\in \Omega$, $p\in K$, and $\tau\in\{\tau\}_\lambda$.
Then, using the inequality $\sqrt{a+b}\leq \sqrt{a}+\sqrt{b}$, we get
\begin{align*}
\sqrt{W(y+\tau, p)} - \sqrt{W(y, p)} &=
    \sqrt{|W(y+\tau, p) - W(y, p)+W(y, p)|} - \sqrt{W(y, p)} \\
&\leq \sqrt{|W(y+\tau, p) - W(y, p)| + W(y, p)} - \sqrt{W(y, p)} \\
&\leq \sqrt{|W(y+\tau, p) - W(y, p)|} \\
&\leq \lambda.
\end{align*}
The opposite inequality follows in a similar way.
Therefore, we get that
\[
\left| \sqrt{W(y+\tau, p)} - \sqrt{W(y, p)}  \right| \leq \lambda.
\]
Since $\lambda>0$, $y\in\R$ and $p\in K$ were arbitrary, this concludes the proof.
\end{proof}

We now use the above result to prove that $\dw$ is Bohr almost periodic.

\begin{lemma}\label{lem:dW_Sap}
Assume {\rm \ref{H1}}-{\rm \ref{H5}} hold.
Then, for all $p,q\in\R^M$, the function $y\mapsto\dw(y;p,q)$ is Bohr almost periodic.
\end{lemma}

\begin{proof}
Fix $p,q\in\R^M$, and let $\lambda>0$.
Using Lemma \ref{lem:geodesics}, we know that there exists $R>0$ such that any solution $\gamma_y$ for the minimization problem defining $\dw(y;p,q)$ takes values in $B_R(0)$, for all $y\in\R$.
Moreover, there exists $L>0$ such that
\begin{equation}\label{eq:bound_L}
\int_{-1}^1 |\gamma_y'(t)|\dif t \leq L,
\end{equation}
for all $y\in\R$.
Using \ref{H1}, there exists a relatively dense set $\{\tau\}_\lambda$ be such that
\[
\sup_{y\in\R} \left| \sqrt{W(y+\tau,p)} - \sqrt{W(y,p)} \right|<\lambda,
\]
for any $\tau\in\{\tau\}_\lambda$ and for any $p\in \R^M$ with $|p|\leq R$.
Let $\gamma\in W^{1,1}([-1,1];\R^M)$ be such that $\gamma(-1)=p$, $\gamma(1)=q$, and $\gamma(t)\in B_R(0)$ for all $t\in[-1,1]$.
Then, for all $\tau\in\{\tau\}_\lambda$ it holds that
\begin{align*}
\int_{-1}^1 \sqrt{W(y+\tau, \gamma(t))}|\gamma'(t)|\dif t - \dw(y;p,q)
    &\leq \int_{-1}^1 \sqrt{W(y+\tau, \gamma(t))}|\gamma'(t)|\dif t -
        \int_{-1}^1 \sqrt{W(y, \gamma(t))}|\gamma'(t)|\dif t \\
&\leq \int_{-1}^1 \left| \sqrt{W(y+\tau, \gamma(t))}|\gamma'(t)|
            - \sqrt{W(y, \gamma(t))}\right| |\gamma'(t)| \dif t \\
&\leq \lambda \int_{-1}^1 |\gamma'(t)| \dif t \\
&\leq \lambda L,
\end{align*}
where in the last step we used \eqref{eq:bound_L}.
Taking the infimum over all admissible curves, we get that
\[
\dw(y+\tau;p,q) - \dw(y;p,q)\leq \lambda L.
\]
The other inequality is obtained in a similar way.
\end{proof}


\section{Compactness}
\label{sec:Compactness}

In this section, we are going to prove Theorem \ref{thm:compactness}.

\begin{proof}
    Let $(u_n)_{n } \subset H^{1}(\Omega; \R^M)$ be a sequence such that
    $$ \sup_{n \in \N} F_{\e_n, \delta_n} (u_n) \leq C < +\infty. $$
    Let $R > 0$ be the constant given in Assumption \ref{H3}. Up to increasing $R$, we can assume $R > \max\{|z_1|, \dots, |z_k| \}$.  Let $\varphi \in C^\infty ([0,\infty))$ be a function such that $\varphi \equiv 1$ for $t \in [0,R)$ and $\varphi \equiv 0$ for $t \geq 2R$. Taking into account the potential $V$ defined as in Assumption \ref{H4}, we introduce the following modified potential $\mathcal{W} \colon \R^M \to [0,\infty)$ defined as
    $$ \mathcal{W}(z) \coloneqq \varphi (|z|) V (z) + (1 - \varphi(|z|) ) \frac{1}{R} |z|. $$
    The potential ${\mathcal W}$ satisfies Assumptions \ref{H3} and\ref{H5} and  vanishes only at the points $\{z_1, \dots, z_k\}$. Since $\mathcal{W} (z) \leq W (x,z)$ for any $x \in \R$ and $p\in \R^M$, we get
    $$ \sup_{n \in \N} \int_\Omega \left[ \frac{1}{\varepsilon_n}  \mathcal{W} (u_n) + \varepsilon_n | \nabla u_n |^2 \right] \dif x \leq \sup_{n \in \N} F_{\e_n, \delta_n} (u_n) \leq C < +\infty. $$
    This allows us to use the methods of \cite[Section 4]{fonseca1989gradient} in order to extract a subsequence $(u_{n_k})_{k } \subset H^{1}(\Omega; \R^M)$ such that $u_{n_k} \to u \in \mathrm{BV}(\Omega; \{z_1,\dots, z_k\})$ strongly in $L^1 (\Omega; \R^M)$. 
\end{proof}


\section{Liminf inequality}
\label{sect:liminf}

The goal of this section is to prove the following result.

\begin{prop}\label{prop:liminf}
Let $u\in L^1(\Omega;\R^M)$, and let $(\e_n)_n, (\delta_n)_n\subset(0,\infty)$ infinitesimal sequences such that
\[
\lim_{n\to\infty} \frac{\varepsilon_n}{\delta_n}=0.
\]
Assume that {\rm \ref{H1}}-{\rm \ref{H5}} hold.
Then,
\[
F_{0,0}(u) \leq \liminf_{n\to\infty} F_{\e_n,\delta_n}(u_n),
\]
for all sequence  $(u_n)_n\subset L^1(\Omega;\R^M)$ such that $u_n\to u$ with respect to the strong  $L^1(\Omega;\R^M)$ topology.
\end{prop}

\begin{proof}
Let $(u_n)_n\subset H^1(\Omega;\R^M)$ with $u_n\to u$ in $L^1(\Omega; \R^M)$. 
Without loss of generality, we can assume that
\begin{equation}\label{eq:liminf=lim}
\Lambda\coloneqq \liminf_{n\to\infty} F_{\e_n,\delta_n}(u_n)
	= \lim_{n\to\infty} F_{\e_n,\delta_n}(u_n) < \infty.
\end{equation}
Since some of the steps in our arguments are nowadays standard in the community, we prefer to focus on the novelties of our argument, and spare the reader with the technical details of such standard steps.

Using a standard truncation argument (see, for instance, Step 1 of the proof of \cite[Theorem 4.1]{CFG_fixed_sup}), we can further assume that there exists a compact set $K\subset\R^M$ such that
\begin{equation}\label{eq:u_n_in_K}
u_n(x)\in K \quad\quad\quad\text{ for a.e. } x\in\Omega.
\end{equation}
For more details, the reader can consult, for instance, Step 1 of the proof of \cite[Proposition 42]{CriPig}. \\

\emph{Step 1: Set-up.} 
We will use the blow-up method.
For any $n\in\mathbb{N}$, consider the finite nonnegative Radon measure
     \begin{equation}
         \notag
         \mu_n\coloneqq \left[ {1\over \e_n} W\left({x\over \delta_n}, u_n(x) \right)  + \e_n |u'_n(x)|^2 \right] \mathcal{L}^1 \restr \Omega.
     \end{equation}
Thanks to \eqref{eq:liminf=lim}, it follows that
\[
\sup_{n\in\mathbb{N}}\mu_n(\Omega)<\infty.
\]
Therefore, up to a subsequence, $\mu_n$ weakly$^\ast$ converges to some finite nonnegative Radon measure $\mu$ in $\Omega$.
By the compactness result Theorem \ref{thm:compactness}, we have that $u\in BV(\Omega; \{z_1, \dots z_k\})$.

Let $x_0\in\Omega\cap J_u$ be a jump point of $u$.
Without loss of generality, we can assume that $u^-(x_0)=z_1$ and that $u^+(x_0)=z_2$.
We claim that 
    \begin{equation}\label{eq:blow_up}
        {d\mu\over d\mathcal{H}^0 \restr J_u}(x_0) \geq \sigma_{12}. 
    \end{equation}
Since $\mu_n$ and $\mu$ are non-negative measures, this will allow us to conclude.

We are now left with proving \eqref{eq:blow_up}.
A standard argument from measure theory gives the existence of an infinitesimal sequence $(r_k)_k\subset(0,1)$ such that
\begin{equation}\label{eq:liminf_1}
{d\mu\over d\mathcal{H}^0 \restr J_u}(x_0)
    = \lim_{k\to\infty} \lim_{n\to\infty}
        \int_{x_0-r_k}^{x_0+r_k} \left[ {1\over \e_n} W\left({x\over \delta_n}, u_n(x) \right)  + \e_n |u'_n(x)|^2 \right] \dx.
\end{equation}
Using the change of variables $x=x_0+r_k y$, and defining
\[
\widetilde{u}_{k,n}(y)\coloneqq u_n(x_0+r_k y),
\]
we have that
\begin{equation}\label{eq:liminf_2}
\int_{x_0-r_k}^{x_0+r_k} \left[ {1\over \e_n} W\left({x\over \delta_n}, u_n(x) \right)  + \e_n |u'_n(x)|^2 \right] \dx
    =
\int_{-1}^{1} \left[ {r_k\over \e_n} W\left( \frac{x_0}{\delta_n} + {r_k\over \delta_n}y, \widetilde{u}_{n,k}(y) \right)  + \frac{\e_n}{r_k} |\widetilde{u}'_{k,n}(y)|^2 \right] \dy,
\end{equation}
where we used the fact that $\widetilde{u}'_{k,n}(y) = r_k u'_n(x_0+r_k y)$.
Using the fact that $u^-(x_0)=z_1$ and of $u^+(x_0)=z_2$, we obtain that
\begin{equation}\label{eq:conv_L1}
\lim_{k\to\infty} \lim_{n\to\infty} \| \widetilde{u}_{k,n} - \varphi\|_{L^1((-1,1); \R^M)} = 0,
\end{equation}
where
\[
\varphi(x) \coloneqq
\left\{
\begin{array}{ll}
z_1 & \text{ if } x\leq 0,\\
z_2 & \text{ if } x>0.
\end{array}
\right.
\]
Now, a diagonal argument due to Attouch (see \cite[Corollary 1.16]{A84}) allows us to find an increasing sequence $(n_k)_k$ such that
\begin{equation}\label{eq:diagonal}
\lim_{k\to\infty} \frac{\varepsilon_{n_k}}{\delta_{n_k}}=0,
\quad\quad\quad
\lim_{k\to\infty} \frac{\varepsilon_{n_k}}{r_k} = \lim_{k\to\infty} \frac{\delta_{n_k}}{r_k}=0,
\end{equation}
and (recall \eqref{eq:conv_L1})
\[
\lim_{k\to\infty} \| \widetilde{v}_k - \varphi \|_{L^1((-1,1); \R^M)}=0,
\]
where, we define, for $k\in\N$, the function $\widetilde{v}_k:(-1,1)\to\R^M$ as $\widetilde{v}_k\coloneqq \widetilde{u}_{k,n_k}$.\\


\emph{Step 2: Change of the boundary conditions.} 
We claim that there exists $(v_k)_k\subset H^1((-1,1);\R^M)$ with $v_k(-1) = z_1$, and $v_k(1)= z_2$, such that
\begin{align}\label{eq:liminf_3}
&\liminf_{k\to\infty} \int_{-1}^{1} \left[ {r_k\over \e_n} W\left( \frac{x_0}{\delta_n} + {r_k\over \delta_n}y, \widetilde{u}_{n,k}(y) \right)
    + \frac{\e_n}{r_k} |\widetilde{u}'_{k,n}(y)|^2 \right] \dy \notag\\
&\hspace{3cm}\geq \liminf_{k\to\infty} \int_{-1}^{1} \left[ {r_k\over \e_{n_k}} W\left( \frac{x_0}{\delta_{n_k}} + {r_k\over \delta_{n_k}}y, v_k(y) \right) 
        + \frac{\e_{n_k}}{r_k} |v'_{k}(y)|^2 \right] \dy,
\end{align}
and such that $v_k\to \varphi$ with respect to the strong $L^1((-1,1);\R^M)$ topology.
This result is standard, known as De Giorgi's slicing. A proof in this setting can be found, for instance, in \cite[Lemma 3.1]{CriFonHagPop}. Note that we can refer to that result because the argument used in there only relies on the growth of the potential at infinity, and not the relation between the parameters $\delta_{n_k}$ and $\varepsilon_{n_k}$.\\


\emph{Step 3: Introduction of an auxiliary potential.}
Let $K\subset\R^M$ be the compact set given in \eqref{eq:u_n_in_K}.
Fix $\eta>0$. Thanks to Definition \ref{def:Bohralmper}, there exists a set $\{\tau_i\}_{i\in I_\lambda}$ relatively dense such that 
          \begin{align}
              \notag
             \sup_{x\in\R} |W(x+\tau_i, p) - W(x, p)|\leq \eta, 
          \end{align}
 for any $p\in K$ and for any $i\in I_\lambda$. Set 
          \begin{align}
              \notag
              \widetilde{W}_\eta (x, p) \coloneqq \inf_{i\in I_\lambda} W(x+\tau_i, p).
          \end{align}
Note that in case $W$ is periodic in the first variable, it holds that $\widetilde{W}_\eta=W$ for all $\eta>0$.
We claim that the following hold:
  \begin{itemize}
      \item[(i)] For any $x\in\Omega$ and for any $p\in\R^M$, it holds that
      \[
      \widetilde{W}_\eta(x,p)\geq V(p),
      \]
      where $V$ given as in {\rm{\ref{H4}}};
      \item[(ii)] Let $R>0$ be the constant given in {\rm{\ref{H3}}}. Then,
      \[
      \widetilde{W}_\eta(x,p)\geq {|z|\over R},
      \]
      for any $x\in\Omega$ and for any $z\in\R^M$, with $|z|\geq R$;
      \item[(iii)] We have that $\widetilde{W}_\eta(x,p)=0$ if and only if $p\in \{z_1, \dots, z_k\}$;
      \item[(iv)] The function $(x, p)\mapsto \widetilde{W}_\eta(x,p)$ is continuous.
  \end{itemize}
Conditions (i) and (ii) follow directly from {\rm{\ref{H3}}} and {\rm{\ref{H4}}}, respectively.
In order to prove (iii), we note that $\widetilde{W}_\eta(x,p)=0$ if $p\in \{z_1, \dots, z_k\}$.
Using (ii), also the reverse implication holds, as desired.

Finally, we prove (iv). Let $(\overline{x},\overline{p})\in \R\times\R^M$.
Assume by contradiction that there exists $(x_n,p_n)_n\subset \R\times\R^M$ with $(x_n,p_n)\to (\overline{x},\overline{p})$ such that
\[
\left| \lim_{n\to\infty} \widetilde{W}_\eta(x_n,p_n) - \widetilde{W}_\eta(\overline{x},\overline{p}) \right| >0.
\]
Without loss of generality, we can assume that
\[
\lambda\coloneqq \lim_{n\to\infty} \widetilde{W}_\eta(x_n,p_n) - \widetilde{W}_\eta(\overline{x},\overline{p})  >0.
\] 
Using the uniform continuity of $W$ in $\R\times \overline{B(\overline{x}, 1)}$ given by Lemma \ref{lem:unif_cont}, there exists $\mu>0$ such that
\[
|x-y|< \mu, |p-q|<\mu \,\Rightarrow \, \left| W(x,p) - W(y,q) \right| < \frac{\lambda}{2}.
\]
Thus, there exists $\overline{n}\in\N$ such that, for all $n\geq\overline{n}$, it holds that
\[
\left| W(x_n,p_n) - W(\overline{x},\overline{p}) \right| < \frac{\lambda}{2}.
\]
Therefore, for all $i\in I_\lambda$, and all $n\geq\overline{n}$, we get
\[
\widetilde{W}_\eta(x_n,p_n) \leq W(x_n+\tau_i,p_n) \leq W(\overline{x},\overline{p}) + \frac{\lambda}{2}.
\]
Taking the infimum over all $i\in I_\lambda$, shifting the terms around, and taking the limit in $n$ yields
\[
0 < \lambda = \lim_{n\to\infty} \widetilde{W}_\eta(x_n,p_n) - \widetilde{W}_\eta(\overline{x},\overline{p})  < \frac{\lambda}{2},
\]
which gives the desired contradiction.

For $x\in\R$, and $p,q\in\R^M$, define
\[
\mathrm{d}_{\widetilde{W}_\eta}(x;p,q) \coloneqq \inf\left\{ 2\int_{-1}^1 \sqrt{\widetilde{W}_\eta(x,\gamma(t))}|\gamma'(t)|\, \dif t \,:\, \gamma\in W^{1,1}([-1,1];\R^M),
    \gamma(-1)=p,\, \gamma(1)=q \right\},
\]
and, for $i\neq j\in\{1,\dots,k\}$,
    \[
    \widetilde{\sigma}^\eta_{ij}\coloneqq \inf\left\{ \int_{J_w} \mathrm{d}_{\widetilde{W}_\eta}(x; w^{-}(x), w^{+}(x)) d{\mathcal H}^0(x) 
        \,:\, w\in \mathcal{A}_{i,j}  \right\},
    \]
where the admissible class $\mathcal{A}_{i,j}$ is defined as in Definition \ref{def:limiting_density}.\\


\emph{Step 4: Approximate estimate of the energy.} 
We claim that 
\[
{d\mu\over d\mathcal{H}^0 \restr J_u}(x_0) \geq \widetilde{\sigma}^\eta_{12}.
\]
We split the proof in three steps.

\emph{Step 4.1: the cell formula.} The relatively density of $\{\tau_i\}_{i\in I_\lambda}$ implies that, for each $k\in\N$, there exists $\tau_{n_k}\in \{\tau_i\}_{i\in I_\lambda}$ such that
          \begin{equation}\label{eq:tau}
              {x_0\over \delta_{n_k}} = \tau_{n_k} -s_{n_k}, \;\;\;\ \hbox{with } |s_{n_k}|<\ell_\eta.
          \end{equation}
Then,
    \begin{align}\label{eq:liminf_4}
    &\liminf_{k\to\infty} \int_{-1}^{1} \left[ {r_k\over \e_{n_k}} W\left( \frac{x_0}{\delta_{n_k}} + {r_k\over \delta_{n_k}}y, v_k(y) \right) 
        + \frac{\e_{n_k}}{r_k} |v'_{k}(y)|^2 \right] \dy\notag\\
    &\quad=\liminf_{k\to\infty} \int_{-1}^{1} \left[ {r_k\over \e_{n_k}} W\left( \tau_{n_k} -s_{n_k} + {r_k\over \delta_{n_k}}y, v_k(y) \right) 
        + \frac{\e_{n_k}}{r_k} |v'_{k}(y)|^2 \right] \dy\notag\\
        &\quad\geq\liminf_{k\to\infty} \int_{-1}^1 \left[ {r_k\over \e_{n_k}}\widetilde{W}_\eta\left(-s_{n_k}+{r_k\over \delta_{n_k}}y, v_k(y)   \right)+ {\e_{n_k}\over r_k} |v'_k(y)|^2\right]\dy \notag \\
        & \quad = \liminf_{k\to\infty}\int_{-1-{\delta_{n_k}\over r_k}s_{n_k}}^{1-{\delta_{n_k}\over r_k}s_{n_k}} \left[{r_k\over \e_{n_k}}\widetilde{W}_\eta\left({r_k\over \delta_{n_k}}z, v_{k}\left(z+{\delta_{n_k}\over r_k}s_{n_k}\right) \right)+ {\e_{n_k}\over r_k} \left|v'_{k}\left(z+{\delta_{n_k}\over r_k}s_{n_k}\right)\right|^2\right]\mathrm{d}z,
    \end{align}
where the last step follows by performing now the change of variables $y= z+{\delta_{n_k}\over r_k}s_{n_k}$.
Now, for each $k\in\N$, let
\[
Q_k\subset \left(-1-{\delta_{n_k}\over r_k}s_{n_k},1-{\delta_{n_k}\over r_k}s_{n_k}\right)
\]
be the biggest interval centered at the origin with that property.
Using \eqref{eq:diagonal} and \eqref{eq:tau}, we get that $Q_k$ converges to $(-1,1)$, as $k\to\infty$.
Moreover, set
\[
w_k(z)\coloneqq v_k\left(z+{\delta_{n_k}\over r_k }s_{n_k}\right).
\]
We obtain 
      \begin{align}\label{eq:liminf_5}
          \liminf_{k\to\infty}&\int_{-1-{\delta_{n_k}\over r_k}s_{n_k}}^{1-{\delta_{n_k}\over r_k}s_{n_k}} \left[{r_k\over \e_{n_k}}\widetilde{W}_\eta\left({r_k\over \delta_{n_k}}z, v_{k}\left(z+{\delta_{n_k}\over r_k}s_{n_k}\right) \right)+ {\e_{n_k}\over r_k} \left|v'_{k}\left(z+{\delta_{n_k}\over r_k}s_{n_k}\right)\right|^2\right]\mathrm{d}z\notag\\
          &\quad=\liminf_{k\to\infty}\int_{-1-{\delta_{n_k}\over r_k}s_{n_k}}^{1-{\delta_{n_k}\over r_k}s_{n_k}} \left[{r_k\over \e_{n_k}}\widetilde{W}_\eta\left({r_k\over \delta_{n_k}}z,  w_{k}(z) \right)+ {\e_{n_k}\over r_k} \left| w'_{k}(z)\right|^2\right]\mathrm{d}z\notag\\
          &\quad \geq \liminf_{k\to\infty}\int_{Q_k} \left[{r_k\over \e_{n_k}}\widetilde{W}_\eta\left({r_k\over \delta_{n_k}}z, w_{k}(z)   \right)+ {\e_{n_k}\over r_k} |w'_{k}(z)|^2\right]\mathrm{d}z,\notag\\
          &\quad \geq \liminf_{k\to\infty}\int_{{r_k\over \delta_{n_k}}Q_k} \left[{\delta_{n_k}\over \e_{n_k}}\widetilde{W}_\eta\left(y, \widetilde{w}_{k}(y)   \right)+ {\e_{n_k}\over \delta_{n_k}} |(\widetilde{w}_k)' (y)|^2\right]\mathrm{d}y,
      \end{align}
where, in the last step, we used the change of variable ${r_k\over \delta_{n_k}}z = y$ set $\widetilde{w}_{k}(y)=w_{k}\left({{\delta_{n_k}} \over r_k} y\right)$, and used the fact that
$(\widetilde{w}_k)' (y) =  \delta_{n_k} w'_k(y)/r_k$.\\

\emph{Step 4.2: the limiting configuration.}
Thanks to Step 2, we have that
\[
\widetilde{w}_k\left( -1-{\delta_{n_k}\over r_k}s_{n_k} \right) = z_1,\quad\quad
\widetilde{w}_k\left( 1-{\delta_{n_k}\over r_k}s_{n_k} \right) = z_2,
\]
for all $k\in\N$.
Therefore, we can define the function $f_k\in H^1(\R;\R^M)$ as
\[
f_k(z)\coloneqq
\left\{
\begin{array}{ll}
z_1 & \text{ if } z\leq -1-{\delta_{n_k}\over r_k}s_{n_k},\\
& \\
\widetilde{w}_{k}(z) & \text{ if } z\in \left[-1-{\delta_{n_k}\over r_k}s_{n_k}, 1-{\delta_{n_k}\over r_k}s_{n_k}  \right],\\
& \\
z_2 & \text{ if } z\geq 1-{\delta_{n_k}\over r_k}s_{n_k}.
\end{array}
\right.
\]
We now claim that there exists $f\in \mathcal {A}_{1,2}$ such that, up to a subsequence (not relabeled), $f_k\to f$ strongly in $L^1_{\mathrm{loc}}(\R;\R^M)$.

Indeed, for $L\in\N$, consider the bounded interval $(-L,L)\in\R$.
Then, there exists $\overline{k}\in\N$ such that
\[
(-L,L)\subset {r_k\over \delta_{n_k}}Q_k
\]
for all $k\geq\overline{k}$.
In particular, from \eqref{eq:liminf_5},  \eqref{eq:liminf_4}, \eqref{eq:liminf_3}, \eqref{eq:liminf_2}, \eqref{eq:liminf_1}, and \eqref{eq:liminf=lim}, we get that
\[
\sup_{k\in\N} \int_{(-L,L)} \left[{\delta_{n_k}\over \e_{n_k}}\widetilde{W}_\eta\left(z, f_{k}(z)   \right)+ {\e_{n_k}\over \delta_{n_k}} |f'_k (z)|^2\right]\mathrm{d}z < \infty.
\]
Recalling \eqref{eq:diagonal} and using the compactness part of Theorem \ref{thm:mm_x}, we have that, up to a subsequence (not relabeled), $(f_k)_k$ converges to some $f^L\in BV((-L,L);\{z_1,\dots, z_k\})$.
Using a diagonal argument, we get that there exists $f\in BV_{\mathrm{loc}}(\R;\{z_1,\dots, z_k\})$ such that, up to a not relabeled subsequence, 
$(f_k)_k$ converges to $f$ strongly in $L^1_{\mathrm{loc}}(\R;\R^M)$.

Now, using the $\Gamma$-convergence part of Theorem \ref{thm:mm_x} and \eqref{eq:liminf=lim}, we get that for each $L\in\N$, we have that
\[
\Lambda > \int_{J_f\cap (-L,L)} \mathrm{d}_{\widetilde{W}^\eta}(x;f^-(x),f^+(x)) \dif \ho(x) > C \ho(J_f\cap (-L,L)).
\]
The last step follows from (i) of Step 3. In particular, the constant $C>0$ is independent of $L$.
Therefore, we get that $f$ has a finite number of jump points in the entire $\R$.
In particular, since each $f_k$ is equal to $z_1$ in $(-\infty, a_k)$ and equal to $z_2$ in $(b_k,+\infty)$, for some $a_k<b_k$, we get that $f\in \mathcal {A}_{1,2}$.
This proves the claim.\\

\emph{Step 4.3: getting the approximate estimate.} Using \eqref{eq:blow_up}, \eqref{eq:liminf_1}, \eqref{eq:liminf_2}, \eqref{eq:liminf_3}, \eqref{eq:liminf_4}, \eqref{eq:liminf_5}, we obtain
\[
{d\mu\over d\mathcal{H}^0 \restr J_u}(x_0) \geq 
    \int_{J_f} \mathrm{d}_{\widetilde{W}^\eta}(x;f^-(x),f^+(x)) \dif \ho(x)
    \geq \widetilde{\sigma}^\eta_{12},
\]
where the last step follows from the definition of $\widetilde{\sigma}^\eta_{1,2}$ and Step 4.2.\\

\emph{Step 5: Conclusion.}
We claim that
\[
\lim_{\eta\to0} \widetilde{\sigma}^\eta_{12} = \sigma_{12}.
\]

First of all, since by definition $\widetilde{W}_\eta\leq W$ for all $\eta>0$, we get that
\[
\widetilde{\sigma}^\eta_{12} \leq\sigma_{12}.
\]
In order to prove the opposite inequality, we argue as follows.
Fix $\eta>0$. Let $w\in \mathcal{A}_{12}$ be such that
\begin{equation}\label{eq:liminf_dW}
\int_{J_w} \mathrm{d}_{\widetilde{W}^\eta}(x; w^{-}(x), w^{+}(x)) \dif{\mathcal H}^0(x) \leq \widetilde{\sigma}^\eta_{12}  + \eta.
\end{equation}
We claim that, without loss of generality, we can assume that
\begin{equation}\label{eq:number_jumps}
\ho(J_w) \leq k-1.
\end{equation}
Indeed, let $w\in\mathcal{A}_{12}$ violating the above bound. Then, there exist $x_1< x_2\in\R$ and $i\neq j\in\{1,\dots,k\}$ such that
\[
w^-(x_1) = z_i,\quad\quad
w^+(x_1) = w^-(x_2) = z_j,\quad\quad
w^+(x_2) z_i.
\]
Then, by considering $\widetilde{w}\in\mathcal{A}_{12}$ defined as
\[
\widetilde{w}(z)\coloneqq\left\{
\begin{array}{ll}
w(z) & \text{ if } x\in \R\setminus (x_1,x_2),\\
z_i & \text{ if } x\in [x_1,x_2],
\end{array}
\right.
\]
we get that 
\[
\int_{J_{\widetilde{w}}} \mathrm{d}_{\widetilde{W}^\eta}(x; \widetilde{w}^{-}(x), \widetilde{w}^{+}(x)) \dif{\mathcal H}^0(x) 
    < \int_{J_w} \mathrm{d}_{\widetilde{W}^\eta}(x; w^{-}(x), w^{+}(x)) \dif{\mathcal H}^0(x).
\]
Thus, we will assume that \eqref{eq:number_jumps} holds.

Now, for each $x\in J_w$, we claim that
\begin{equation}\label{eq:claim:dW}
\dw(x; w^{-}(x), w^{+}(x)) \leq \mathrm{d}_{\widetilde{W}^\eta}(x; w^{-}(x), w^{+}(x)) + 2L \sqrt{\eta},
\end{equation}
where $L>0$ is the constant given by Lemma \ref{lem:geodesics} relative to $\max\{|z_i| : i=1,\dots, k \}$.
Note that such a constant is independent of $\eta$.
In order to prove the validity of \eqref{eq:claim:dW}, thanks to Lemma \ref{lem:geodesics}, there exists $\gamma\in W^{1,1}([-1,1];\R^M)$ such that $\gamma(-1)= w^{-}(x)$, $\gamma(1)= w^{+}(x)$ be such that
\begin{equation}\label{eq:est:gamma_liminf}
2\int_{-1}^1 \sqrt{\widetilde{W}^\eta(x,\gamma(t))}|\gamma'(t)|\, \dif t = \mathrm{d}_{\widetilde{W}^\eta}(x; w^{-}(x), w^{+}(x)).
\end{equation}
Now, we write
\begin{align*}
2\int_{-1}^1 \sqrt{W(x,\gamma(t))}|\gamma'(t)|\, \dif t 
    &\leq 2\int_{-1}^1 \sqrt{\widetilde{W}^\eta(x,\gamma(t))}|\gamma'(t)|\, \dif t
        + 2\int_{-1}^1 \left| \sqrt{\widetilde{W}_\eta(x,\gamma(t))} - \sqrt{W(x,\gamma(t))} \right|  \,|\gamma'(t)|\, \dif t \\
    &\leq 2\int_{-1}^1 \sqrt{\widetilde{W}^\eta(x,\gamma(t))}|\gamma'(t)|\, \dif t
        + 2 \sqrt{\eta} \int_{-1}^1 |\gamma'(t)|\, \dif t \\
    &\leq \mathrm{d}_{\widetilde{W}^\eta}(x; w^{-}(x), w^{+}(x)) + 2 \sqrt{\eta} L,
\end{align*}
where the previous to last step follows from Lemma \ref{lem:sqrtW_Sap} , while last step by using Lemma \ref{lem:geodesics} and \eqref{eq:est:gamma_liminf}.
This gives the desired claim.

Now, from \eqref{eq:liminf_dW} and \eqref{eq:claim:dW}, we get that
\begin{align*}
\sigma_{12}&\leq \int_{J_w} \dw(x; w^{-}(x), w^{+}(x)) \dif{\mathcal H}^0(x) \\
&\leq \int_{J_w} \mathrm{d}_{\widetilde{W}^\eta}(x; w^{-}(x), w^{+}(x)) \dif{\mathcal H}^0(x)
        + 2\ho(J_w)\sqrt{\eta}L \\
&\leq \widetilde{\sigma}^\eta_{12}  + \eta+ 2\ho(J_w)\sqrt{\eta}L.
\end{align*}
Since $\eta>0$ is arbitrary, taking into consideration \eqref{eq:number_jumps}, we get the validity of the opposite inequality.

This concludes the proof.
\end{proof}


\section{Limsup inequality}
\label{sect:limsup}

This section is fully devoted to proving that that lower bound obtained in the previous section is actually optimal.
Namely, we will prove the following result.

\begin{prop}\label{prop:limsup}
Assume $(\e_n)_n, (\delta_n)_n\subset(0,\infty)$ infinitesimal sequences such that
\[
\lim_{n\to\infty}\frac{\varepsilon_n}{\delta_n}=0.
\]
Assume that {\rm \ref{H1}}-{\rm \ref{H5}} hold.
Let $u\in BV(\Omega;\{z_1, \dots, z_k\})$.
Then, there exists a sequence $(u_n)_n\subset H^1(\Omega;\R^M)$ such that
\[
\limsup_{n\to\infty} F_{\e_n,\delta_n}(u_n) \leq F_{0,0}(u),
\]
and with $u_n\to u$ with respect to the strong  $L^1(\Omega;\R^M)$ topology.
\end{prop}

\begin{proof}
Let $u\in BV(\Omega;\{z_1, \dots, z_k\})$.
Then, $u$ has only finitely many jump points $x_0<x_1<\dots<x_p\in\Omega$.
We claim that, for each $i\in\{1,\dots,p\}$ and $n\in\N$, there exist $L^i_n>0$, and $u^i_n\in H^1(x_i-L^i_n, x_i+L^i_n)$ such that the sequence $(u_n)_n\subset H^1(\Omega;\R^M)$ where
\[
u_n(x)\coloneqq u(x) \quad\text{ if } x\in \Omega\setminus \bigcup_{i=1}^p (x_i-L^i_n, x_i+L^i_n),\quad\quad\quad
u_n(x)\coloneqq u^i_n(x)  \quad\text{ if } x\in (x_i-L^i_n, x_i+L^i_n)
\]
satisfies the required property.

The way we will build the functions $u^i_n$ will only depend on $x_i$, $u^-(x_i)$ and $u^+(x_i)$.
Therefore, without loss of generality, we will provide the construction in the case $p=1$, and $u^-(x_i)=z_1$ and $u^+(x_i)=z_2$.
We divide the proof in several steps.\\

\emph{Step 1: Choice of the jump point and of the curve.}
Fix $\lambda>0$.
Let $\bar{y}\in \R$ be such that
\begin{equation}\label{eq:choice_y}
\dw(\bar{y};z_1, z_2) \leq \sigma_{12}(z_1, z_2) + \lambda.
\end{equation}
Let $\gamma\in W^{1,1}([-1,1];\R^M)$ with $\gamma(-1)=z_1$, and  $\gamma(1)=z_2$ be such that
\begin{equation}\label{eq:choice_gamma}
2\int_{-1}^1 \sqrt{W(\bar{y},\gamma(t))}|\gamma'(t)|\,\mathrm{d}t \leq \dw(\bar{y};z_1, z_2) + \lambda.
\end{equation}
Without loss of generality, we can assume that $\gamma\in C^1([-1,1];\R^M)$, and that $|\gamma'(t)|\neq 0$ for all $t\in[-1,1]$.
We extend $\gamma$ to be defined on the entire $\R$, by setting (with an abuse of notation) $\gamma(t)\coloneqq z_1$ for $t\leq -1$, and $\gamma(t)\coloneqq z_2$ for $t\geq 1$.
Let $K\subset\R^M$ be a compact set such that $\gamma(t)\in K$ for all $t\in[-1,1]$, and let
\begin{equation}\label{eq:length_L}
L\coloneqq \int_{-1}^1 |\gamma'(t)|\,\mathrm{d}t <\infty,
\end{equation}
which is finite thanks to Lemma \ref{lem:geodesics}.\\

\emph{Step 2: Construction of the recovery sequence.}
By Lemma \ref{lemma:reparameterization}, for any $n\in \mathbb{N}$, let $\nu_n>0$ be constant satisfying \eqref{eq:nu} and let $g_n\in C^1([0,\nu;_n];[-1,1])$ be the function defined in \eqref{eq:g} and corresponding to
\[
\varepsilon = \varepsilon_n,\quad\quad
\gamma \text{ as above },\quad\quad\quad
\lambda \text{ as above}.
\]

Let $\{\tau_i\}_{i\in I_\lambda}$ be the relative dense set for $W$ and the compact set $K$ from Step 1. Denote by $l>0$ its $\lambda$-uniformly almost period.
For any $n\in\N$, let $\tau_i^n \in \{\tau_i\}_{i\in I_\lambda}$ be such that
\begin{equation}\label{eq:choice_tau_i}
\left|\frac{x_0}{\delta_n} - \tau_i^n\right|  \leq l.
\end{equation}
For any $n\in\mathbb{N}$, we define $u_n\in H^1(\Omega;\R^M)$ as
    \begin{align}
        \notag
        u_n (x) \coloneqq \gamma (g_n(x-\delta_n(\overline{y} + \tau_i^n))).
    \end{align}
Note that, thanks to \eqref{eq:choice_tau_i}, we can estimate
\begin{equation}\label{eq:conv_points}
\left| \delta_n(\overline{y} + \tau_i^n) - x_0 \right|
    \leq \delta_n \left| \bar{y} - \frac{x_0}{\delta_n} + \tau_i^n  \right|
    \leq \delta_n \left(\left| \bar{y} \right| + \left|\frac{x_0}{\delta_n} - \tau_i^n\right| \right )
    \leq \delta_n (\left| \bar{y} \right| + l),
\end{equation}
which vanishes as $n\to\infty$.\\
 
\emph{Step 3: Convergence of the functions.} We are going to prove that $(u_n)_n$ strongly converges in $L^1$ to $u$. Indeed, using \eqref{eq:conv_points}, we have
     \begin{align*}
         \int_{a}^b |u_n(x) - u(x)|dx
         &= \int_{\delta_n(\overline{y}+\tau_i^n)}^{\delta_n(\overline{y}
            +\tau_i^n)+\nu_n} |u_n(x) - u(x)| \dif x
         \leq 2R \nu_n,
     \end{align*}
where $R>0$ is such that $|z_i|\leq R$ for all $i\in\{1,\dots, k\}$, and $K\subset B(0,R)$.\\

\emph{Step 4: Asymptotic upper bound of the energies.}
First of all, we note that for all
\[
x\in [\delta_n(\overline{y}+\tau_i^n), \delta_n(\overline{y}+\tau_i^n)+\nu_n],
\]
we have
\[
\frac{x}{\delta_n} - \bar{y} \in \left[ \tau_i^n, \tau_i^n + \frac{\nu_n}{\delta_n} \right]
    \subset \left[ \tau_i^n, \tau_i^n + \frac{\varepsilon_n}{\delta_n} \frac{L}{\sqrt{\lambda}} \right].
\]
Note that the length of the interval on the left-hand side vanishes as $n\to\infty$.
Therefore, using the uniform almost periodicity of the potential, together with the fact that it is uniformly continuous on $\R$, there exists $\bar{n}\in\N$ such that
\begin{align*}
\left| W\left({x\over \delta_n}; u_n(x) \right) - W\left(\overline{y}; u_n(x) \right) \right|
    &\leq \left| W\left(\overline{y} + \tau^n_i; u_n(x) \right) - W\left(\overline{y}; u_n(x) \right) \right| \\
    &\hspace{2cm } +\left| W\left({x\over \delta_n}; u_n(x) \right) - W\left(\overline{y} + \tau^n_i; u_n(x) \right) \right| \\
    &\leq 2\lambda,
\end{align*}
for all $n\geq\bar{n}$.
Using Lemma \ref{lemma:reparameterization} together with \eqref{eq:choice_y} and \eqref{eq:choice_gamma}, we have that 
    \begin{align*}
        \int_{a}^b {1\over \e_n}W\left({x\over \delta_n}; u_n(x) \right) 
            &+ \e_n |u_n'(x)|\dif x \\
        &= \int_{\delta_n(\overline{y}+\tau_i^n)}^{\delta_n(\overline{y}+\tau_i^n)+\nu_n} {1\over \e_n}W\left({x\over \delta_n}; u_n(x) \right)  + \e_n |u_n'(x)|\dif x\notag\\
         &= \int_{\delta_n(\overline{y}+\tau_i^n)}^{\delta_n(\overline{y}+\tau_i^n)+\nu_n}{1\over \e_n}\left[W\left({x\over \delta_n}; u_n(x) \right) - W\left(\overline{y}; u_n(x) \right)\right] \dif x\notag\\
        &\hspace{3cm} +\int_{\delta_n(\overline{y}+\tau_i^n)}^{\delta_n(\overline{y}+\tau_i^n)+\nu_n}{1\over \e_n}W\left(\overline{y}; u_n(x) \right)+ \e_n |u_n'(x)|\dif x\notag\\
        &\leq {2\lambda\over \e_n}\nu_n + \int_{\delta_n(\overline{y}+\tau_i^n)}^{\delta_n(\overline{y}+\tau_i^n)+\nu_n}{1\over \e_n}W\left(\overline{y}; u_n(x) \right)+ \e_n |u_n'(x)|\dif x\notag\\
        &\leq 2\sqrt{\lambda} L + \int_{\delta_n(\overline{y}+\tau_i^n)}^{\delta_n(\overline{y}+\tau_i^n)+\nu_n}{1\over \e_n}W\left(\overline{y}; u_n(x) \right)+ \e_n |u_n'(x)|\dif x\notag\\
        &\leq 4\sqrt{\lambda} L + \int_{-1}^1 2\sqrt{W(\overline{y}, \overline{\gamma}(t))}|\overline{\gamma}'(t))|\dif t\notag\\
        &\leq 4\sqrt{\lambda} L+2\lambda+\sigma_{1,2}.
    \end{align*}
Since $\lambda>0$ is arbitrary, this concludes the proof.
\end{proof}


\section{The mass-constrained case}
\label{sect:mass-constrained}

In this section, we explain how to adapt our arguments in order to prove Theorem \ref{thm:mass_constrained}.
Let us first consider the pre-compactness.
Let $(u_n)_n\subset H^1(\Omega;\R^M)$ be such that
\[
\sup_{n\in\N} F^m_{\varepsilon_n,\delta_n}(u_n) < +\infty.
\]
Using the fact that $F^m_{\varepsilon_n,\delta_n} \geq F_{\varepsilon_n,\delta_n}$, thanks to Theorem \ref{thm:compactness} we obtain the existence of a subsequence (not relabeled) that $u_n\to u$ strongly in $L^1(\Omega;\R^M)$, for some $u\in BV(\Omega;\{z_1,\dots,z_k\})$
In particular,
\[
\int_\Omega u(x) \dx = \lim_{n\to\infty} \int_\Omega u_n(x) \dx = m.
\]

We now consider the $\Gamma$-convergence result.
Fix $u\in BV(\Omega;\{z_1,\dots,z_k\})$ be such that
\[
\int_\Omega u(x) \dx = m.
\]
For the liminf inequality, let $(u_n)_n\subset H^1(\Omega;\R^M)$ be such that $u_n\to u$ in $L^1(\Omega; \R^M)$. 
Without loss of generality, we we can assume that
\[
\int_\Omega u_n(x) \dx = m.
\]
Thus, $F^m_{\varepsilon_n,\delta_n}(u_n) =  F_{\varepsilon_n,\delta_n}(u_n)$ for all $n\in\N$. Therefore, using Proposition \ref{prop:liminf}, we obtain that
\[
F^m_{0,0}(u) = F_{0,0}(u) \leq \liminf_{n\to\infty} F_{\e_n,\delta_n}(u_n) = \liminf_{n\to\infty} F^m_{\e_n,\delta_n}(u_n).
\]
We now consider the limsup inequality.
Let $(u_n)_n\subset H^1(\Omega;\R^M)$ be the sequence given by Proposition \ref{prop:limsup}.
Set
\[
\int_\Omega u_n(x) \dx = m_n.
\]
Note that, by construction of the function $u_n$, we have that
\begin{equation}\label{eq:u_n_eps_n}
|m_n-m| \leq \|u_n - u\|_{L^1} = \sum_{y\in J_u} \int_{B(y, C\varepsilon_n)} |u_n(x) - u(x)| \dx \leq C \varepsilon_n,
\end{equation}
where the last inequality follows from the fact that $u$ and $u_n$ are uniformly bounded.
We now use the strategy developed in \cite[Section 7]{CriPig} to adjust the mass constraint.
Let $y \in \Omega \setminus J_u$.
Let $i\in\{1,\dots,k\}$ be such that $u(y) = z_i$.
Fix $r > 0$ small enough such that $u_n \equiv z_i$ on $B \coloneqq B(y,r)$.
Let $\phi \in C_c^\infty (B(y,r); [0,1])$ be such that
\begin{equation}\label{eq:int_phi}
\int_{B(y,r)} \phi(x) \dx = 1.
\end{equation}
Define
\[
v_n \coloneqq u_n + w_n,
\]
where
\[
w_n (x) \coloneqq (m - m_n) \phi (x).
\]
First of all, we note that $v_n \in H^1(\Omega; \R^M)$.
Moreover, using \eqref{eq:int_phi}, we get that
\[
\int_\Omega v_n \dx = \int_\Omega u_n \dx + \int_\Omega w_n \dx = m_n + (m - m_n) \int_B \phi \dx = m.
\]
We now consider the convergence in configuration. Thanks to \eqref{eq:u_n_eps_n}, it holds that
\[
    \| v_n - u \|_{L^1} \leq \| u - u_n \|_{L^1} + \| v_n - u_n \|_{L^1} = \| u - u_n \|_{L^1} + | m_n - m| \leq \| u - u_n \|_{L^1} + C \varepsilon_n,
\]
which vanishes as $n\to\infty$.
We now show that we have convergence of the energy.
Note that
\[
|w'_n(x)| = |m-m_n|\,|\phi'(x)|.
\]
Thus, we have that
\begin{align*}
|F_{\e_n,\delta_n}(u_n) - F_{\e_n,\delta_n}(v_n)|
    &= \int_{B(y,r)} \left[ \frac{1}{\e_n} W\left(\frac{x}{\delta_n}, v_n(x) \right) + \e_n |v'_n(x)|^2 \right] \dx \\
    &\leq \int_{B(y,r)} \left[ \frac{1}{\e_n} \theta (w_n(x))|w_n(x)| + \e_n |m-m_n|^2 |\phi'(x)|^2 \right] \dx \\
    &\leq C \theta(|m_n-m|) + \e_n |m-m_n|^2 \int_{B(y,r)} |\phi'(x)|^2 \dx,
\end{align*}
which vanishes as $n\to\infty$.
Note that, in the previous to last step we used \ref{H7}.
This gives the desired conclusion.


\subsection*{Acknowledgment}
RC was partially supported under NWO-OCENW.M.21.336, MATHEMAMI - Mathematical Analysis of phase Transitions in HEterogeneous MAterials with Materials Inclusions.
LD acknowledges support of the Austrian Science Fund (FWF) projects 10.55776/ESP1887024.


\bibliographystyle{siam}
\bibliography{bibliography}

\end{document}